\documentclass{amsart}
\usepackage{graphicx} % Required for inserting images
\usepackage{amsthm, amsmath, amssymb, tikz, bm}
\usepackage{mathtools}
\usepackage{xifthen}
\usepackage{comment}
\usepackage[margin=4.5cm]{geometry}
\usepackage[shortlabels]{enumitem}
\usepackage{todonotes}
\usepackage[colorlinks=true,
linkcolor=blue,citecolor=blue,
urlcolor=blue]{hyperref}

\newcommand{\ba}{\backslash}
\newcommand{\romannum}[1]{\romannumeral#1\relax}

\newcommand{\cC}{\mathcal{C}}
\newcommand{\cl}{\mathrm{cl}}

\newtheorem{theorem}{Theorem}[section]
\newtheorem{lemma}[theorem]{Lemma}
\newtheorem{corollary}[theorem]{Corollary}
\newtheorem{proposition}[theorem]{Proposition}

\theoremstyle{definition}

\theoremstyle{remark}

\numberwithin{equation}{section}

\title{Degeneracy: From Graphs to Matroids}

\begin{document}

\author[Bickle]{Allan Bickle}
\address{Department of Mathematics\\
Purdue University, West Lafayette, Indiana}
\email{aebickle@purdue.edu}

\author[Douthitt]{James Dylan Douthitt}
\address{Department of Mathematics\\
Syracuse University, Syracuse, New York}
\email{jddouthi@syr.edu}

\author[Ge]{Wayne Ge}
\address{Mathematics Department\\
Louisiana State University\\
Baton Rouge, Louisiana}
\email{yge4@lsu.edu}

\author[Singh]{Jagdeep Singh}
\address{Department of Mathematics and Statistics\\
Mississippi State University\\
Starkville, Mississippi}
\email{singhjagdeep070@gmail.com}

\date{\today}

\begin{abstract}
A graph is $k$-degenerate if every subgraph has a vertex of degree at most $k$. We extend this notion to matroids, defining a loopless matroid $M$ to be $k$-degenerate if every restriction of $M$ contains a cocircuit of size at most $k$; $M$ is minimally $k$-degenerate if it has cogirth $k$ and every proper restriction of $M$ has cogirth at most $k-1$. Our main result characterizes extremal minimally $k$-degenerate matroids. We also extend the known arboricity bound for matroids, showing that $k$-degenerate matroids have arboricity at most $k$ and providing sharper bounds.
\end{abstract}

\maketitle

\section{Introduction}

The notation and terminology follow Oxley~\cite{Oxl11}. We call a graph \textit{non-empty} if it has at least one edge. Throughout the paper, all graphs and matroids are finite, loopless, non-empty and may contain parallel elements. In particular, all subgraphs and restrictions considered in the paper are assumed to be non-empty. A graph $G$ is \textit{$k$-degenerate} if every subgraph of $G$ has a vertex of degree at most $k$, and $G$ is \textit{$k$-edge-degenerate} if every subgraph of $G$ has a bond of size at most $k$. Note that $G$ is $k$-edge-degenerate if and only if every subgraph of $G$ has edge-connectivity at most $k$. The \textit{degeneracy} $d(G)$ of a graph $G$ is the smallest integer $k$ such that $G$ is $k$-degenerate. Analogously, the \textit{edge-degeneracy} $ed(G)$ of $G$ is the smallest integer $k$ such that $G$ is $k$-edge-degenerate. Since the set of edges incident to a vertex necessarily contains a bond, it is clear that $ed(G) \leq d(G)$. The concept of edge-degeneracy, introduced by Matula in \cite{matula_72}, is referred to there as the strength of the graph.

While degeneracy in graphs is a well-studied area, the parallel concept of edge-degeneracy is less explored. Transitioning this concept to matroids via cocircuits provides a unifying framework. A loopless matroid $M$ is \textit{$k$-degenerate} if every restriction of $M$ has a cocircuit of size at most $k$. The \textit{degeneracy} $d(M)$ of a loopless matroid $M$ is the smallest integer $k$ such that $M$ is $k$-degenerate. This definition of matroid degeneracy is equivalent to the invariant $D_m(M)$ introduced by Gordon Royle~\cite{royle_2015}. Observe that the cocircuits of the cycle matroid $M(G)$ of a graph $G$ correspond to the bonds of $G$. Therefore, the degeneracy of the cycle matroid $M(G)$ equals the edge-degeneracy of $G$.

The \textit{girth} of $M$, denoted $g(M)$, is the size of a smallest circuit of $M$, while the \textit{cogirth} of $M$, denoted $g^*(M)$, is the size of a smallest cocircuit of $M$. For an integer $k\geq 2$, a loopless matroid $M$ is \textit{minimally $k$-degenerate} if $M$ has cogirth $k$ and every proper restriction of $M$ has cogirth at most $k-1$. 

In Section~\ref{sec: simple}, we provide the following characterization of simple extremal minimally $3$-degenerate matroids. Recall that, for $k\geq 2$, the wheel graph $\mathcal{W}_k$ consists of a $k$-cycle with an additional vertex adjacent to all vertices of the cycle via edges called \textit{spokes}. The whirl matroid $\mathcal{W}^k$ is obtained from $M(\mathcal{W}_k)$ by relaxing a circuit-hyperplane, and its \textit{spokes} are the elements in more than one triangle. A more detailed treatment of wheels and whirls may be found in~\cite[Section~8.4]{Oxl11}. It is worth noting that, every edge of $\mathcal{W}_3$ as well as every element of $\mathcal{W}^2$ is a spoke.

\begin{theorem}\label{thm: simple min 3-degen}
Let $M$ be a simple minimally $3$-degenerate matroid. Then $|E(M)|\leq 2r(M)$.
Moreover, equality is attained if and only if
\begin{enumerate}
    \item[(\romannum{1})] $M$ is the cycle matroid of a wheel $\mathcal{W}_k$ for some $k\geq 3$,
    \item[(\romannum{2})] $M$ is a whirl $\mathcal{W}^k$ for some $k\ge 2$, or
    \item[(\romannum{3})] $M=M_1\oplus_2 M_2$ at a basepoint $p$ where, for each $i\in \{1,2\}$, the matroid $M_i$ is a member of (\romannum{1}) or (\romannum{2}) and $p$ is a spoke in $M_i$.
\end{enumerate}
\end{theorem}

We extend this to the case for simple matroids when $k \geq 4$.

\begin{theorem}\label{thm: simple min k-degen}
For $k \geq 4$, let $M$ be a simple minimally $k$-degenerate matroid. Then $|E(M)| \leq (k-1)(r(M)-1)+2$. Moreover, equality is attained if and only if $M \cong U_{2,k+1}$ or $M \cong U_{2,k+1} \oplus_2 U_{2,k+1}$.
\end{theorem}

We then drop the simplicity requirement to provide the following characterization for general matroids.

\begin{theorem}\label{thm: non-simple min k-degen}
Let $M$ be a minimally $k$-degenerate matroid for some $k\geq 2$. Then $|E(M)|\leq (k-1)r(M)+1$. Moreover, equality is attained if and only if $M$ can be obtained via iterated series connections of copies of $U_{1,k}$. 
\end{theorem}

We also explore the \textit{arboricity}, $a(M)$, of a matroid $M$, defined as the minimum number of independent sets into which $E(M)$ can be partitioned. Minh and Trung~\cite{nguyen} bounded the arboricity of a matroid by its maximum cocircuit size (cocircumference). We extend this by proving that arboricity is bounded by degeneracy. In addition, we characterize the matroids that attain equality in the bound of Minh and Trung.

\begin{theorem}
\label{thm: equality_arboricity_intro}
Let $M$ be a connected matroid. Then the arboricity of $M$ equals the maximum cocircuit size of $M$ if and only if $M$ is isomorphic to $U_{k,k+1}$ or $U_{1,k}$ for some $k \geq 1$.
\end{theorem}

We provide a sharper bound on arboricity using \emph{$e$-cocircumference}, $c^*_e(M)$, the size of a largest cocircuit containing a specific element $e$.

\begin{theorem}
\label{e-circumference_intro}
Let $M$ be a connected matroid with at least two elements. Then $a(M) \leq c^*_e(M)$ for every $e \in E(M)$. 
\end{theorem}

The remainder of the paper is organized as follows. Section~\ref{sec: basic_prop} establishes the basic properties and bounds for $k$-degenerate matroids, including maximum size bounds and the introduction of bi-$k$-degeneracy. Section~\ref{sec: simple} characterizes extremal minimally $k$-degenerate matroids, proving Theorems~\ref{thm: simple min 3-degen}, \ref{thm: simple min k-degen}, and \ref{thm: non-simple min k-degen}. Section~\ref{arboricity} contains the proofs of Theorem~\ref{thm: equality_arboricity_intro} and Theorem~\ref{e-circumference_intro} concerning matroid arboricity. Finally, Section~\ref{sec: graphs} explores the interplay of matroid and graph degeneracies and notes a consequence of Theorem~\ref{thm: simple min 3-degen} for graphs.

\section{Basic Properties: Matroid Analogue of Edge Degeneracy}\label{sec: basic_prop}

Observe that the degeneracy of a graph $G$ is the maximum $\delta(H)$ over all subgraphs $H$ of $G$, where $\delta(H)$ denotes the minimum degree of $H$.
Similarly, the edge-degeneracy of a graph $G$ is the maximum $\lambda(H)$ over all subgraphs $H$, where $\lambda(H)$ is the edge-connectivity of $H$. Analogously, the degeneracy of a matroid $M$ is the maximum cogirth $g^*(N)$ over all restrictions $N$ of $M$. The following proposition establishes several equivalent characterizations of $k$-degeneracy.

\begin{proposition}
\label{matroid_equivalence}
For a matroid $M$, the following are equivalent:
\begin{enumerate}[label=(\roman*)]
    \item $M$ is $k$-degenerate.
    \item For every non-empty flat $F$ of $M$, the restriction $M|F$ has a cocircuit of size at most $k$. 
    \item There is a sequence $F_0, \ldots, F_r$ of flats of $M$ where $r = r(M)$, the flat $F_0$ equals $E(M)$, and $F_r = \emptyset$, such that for each $1 \leq i \leq r$, the flat $F_i$ is obtained from $F_{i-1}$ by deleting a cocircuit $C^*_{i-1}$ of the restriction $M|F_{i-1}$ with $|C^*_{i-1}| \leq k$.
\end{enumerate}
\end{proposition}

\begin{proof}
It is easy to see that (i) $\implies$ (ii) $\implies$ (iii). To prove the result, it is enough to show that (iii) implies (i). For a non-empty subset $A$ of $E(M)$, let $M|A$ be a restriction of $M$, and let $F_0 ,\ldots, F_r$ be a sequence of flats of $M$ as described in (iii). Then there is a $j$ in $\{0, \ldots, r-1\}$ such that $A \subseteq F_j$ and $A \not\subseteq F_{j+1}$. Observe that $C^*_j\cap A \neq \emptyset$ and the set $C^*_j \cap A$ contains a cocircuit of $M|A$. Since $|C^*_j \cap A| \le |C^*_j| \le k$, the restriction $M|A$ has a cocircuit of size at most $k$.
\end{proof}

We call the ordered sequence $C^*_0, \ldots, C^*_{r-1}$ in (iii) a \textit{slicing} of a $k$-degenerate matroid $M$.

It is well-known that the number of edges in a simple $k$-degenerate graph on $n \geq k$ vertices is at most $kn - \binom{k+1}{2}$ (see, for example, \cite{lickwhite}). Using the slicing decomposition, we establish the analogue for $k$-degenerate matroids.  

\begin{proposition}
\label{prop: size_bound}
Let $M$ be a $k$-degenerate matroid of rank $r$. Then $|E(M)|\leq kr$. If $M$ is simple, then $|E(M)| \leq k(r-1)+1$. 
\end{proposition}

\begin{proof}
Let $C^*_0, \ldots, C^*_{r-1}$ be a slicing of $M$.
Since each set in the slicing has size at most $k$, the result follows. If $M$ is simple, the final non-empty flat $F_{r-1}$ must consist of a single element. Thus $|C^*_{r-1}| = 1$. The preceding $r-1$ sets in the slicing have size at most $k$. Therefore $|E(M)| \le k(r-1) + 1$.
\end{proof}

A simple $k$-degenerate graph $G$ is \textit{maximal} if adding any edge between two non-adjacent vertices of $G$ produces a graph that is not $k$-degenerate. A simple \textit{maximal $k$-edge-degenerate} graph is defined similarly. These graphs have been well studied (see, for example, \cite{lai_90, mader_71}). The following result is given by Lick and White \cite[Corollary~1]{lickwhite}. 

\begin{proposition}
\label{prop: lick_white_maximal_k_degen}
If $G$ is a simple maximal $k$-degenerate graph on $n$ vertices with $n\geq k$, then $G$ has $kn-\binom{k+1}{2}$ edges.
\end{proposition} 

Mader~\cite{mader_71} proved the following for maximal $k$-edge-degenerate graphs. 

\begin{proposition}
\label{prop: mader_maximal_k_degen}
The number of edges in a simple $k$-edge degenerate graph on $n \geq k$ vertices is at most $kn - \binom{k+1}{2}$, and the equality is attained if and only if $G$ is maximal $k$-degenerate.
\end{proposition}

Transitioning this to matroids, we distinguish between general and simple cases. We call a matroid $M$ \textit{general maximal $k$-degenerate} if $M$ is $k$-degenerate and no single-element extension of $M$ having the same rank is $k$-degenerate. Using the slicing decomposition, we establish the following.

\begin{proposition}
\label{prop: maximal_non_simple}
Let $M$ be a $k$-degenerate matroid of rank $r$. Then $M$ is general maximal $k$-degenerate if and only if $|E(M)| = kr$. 
\end{proposition}

\begin{proof}
Observe that if $|E(M)| = kr$, then, by Proposition~\ref{prop: size_bound}, $M$ is general maximal $k$-degenerate. 

Conversely, suppose that $M$ is general maximal $k$-degenerate. It suffices to show that every set $C^*_{0}, \ldots, C^*_{r-1}$ in a slicing of $M$ has size equal to $k$. Suppose not. Then for some $t \leq r-1$, we have $|C^*_t| < k$. Note that $C^*_t$ is a cocircuit of the restriction $M|F_t$, where $F_t = E(M) - (C^*_0 \cup \ldots \cup C^*_{t-1})$ is a flat of $M$. Let $M'$ be the principal extension of $M$ with respect to the flat $F_t$, and let $e$ be the newly added element. Observe that $C^*_0, \ldots, C^*_{t-1}, C^*_t \cup \{e\}, C^*_{t+1} \ldots, C^*_{r-1}$ is a slicing of $M'$. Therefore, $M'$ is a $k$-degenerate extension of $M$ that has the same rank as $M$, a contradiction.
\end{proof}

Since the last set $C^*_{r-1}$ in a slicing of a general maximal $k$-degenerate matroid $M$ has size $k$, it follows that $M$ has a parallel class of size $k$. In particular, general maximal $k$-degenerate matroids need not be connected. For instance, a direct sum of copies of $U_{1,k}$ is general maximal $k$-degenerate. 

We now consider simple matroids. For a simple matroid $M$, we say $M$ is \textit{simple maximal $k$-degenerate} if $M$ is $k$-degenerate and no simple single-element extension of $M$ having the same rank is $k$-degenerate. Observe that for $k\geq 2$, the unique simple maximal $k$-degenerate matroid of rank two is $U_{2,k+1}$. Since simple maximal $k$-degenerate matroids of successive higher ranks are constructed from this connected base by iteratively adjoining cocircuits of size $k$, a straightforward inductive argument implies that for $k \geq 2$, every simple maximal $k$-degenerate matroid is connected. 

The proof of Proposition~\ref{prop: maximal_non_simple} can be adapted easily to show the following. 

\begin{proposition}
\label{prop: maximal_simple}
Let $M$ be a simple $k$-degenerate matroid of rank $r$. Then $M$ is simple maximal $k$-degenerate if and only if $|E(M)| = k(r-1) + 1$. 
\end{proposition}

This maximality directly dictates the basis structure of the matroid.

\begin{corollary}
A simple maximal $k$-degenerate matroid $M$ of rank at least two has $k$ bases $B_1, \ldots, B_k$ such that $B_i \cap B_j = \{e\}$ for all distinct $i$ and $j$ and some element $e$ of $M$. Furthermore, if $M$ is a general maximal $k$-degenerate matroid, it possesses $k$ pairwise disjoint bases. 
\end{corollary}

\begin{proof}
Suppose $M$ is simple maximal $k$-degenerate with $r(M)\geq 2$. Let $C^*_0, \ldots, C^*_{r-1}$ be a slicing of $M$, and let $F_0, \ldots, F_r$ be the corresponding sequence of flats such that for $0 \leq i \leq r-1$, the set $C^*_i$ is a cocircuit of $M|F_i$. By Proposition~\ref{prop: maximal_simple}, $C^*_{r-1} = \{e\}$ for some element $e$, and the preceding $r-1$ sets each have size exactly $k$.

For $i$ in $\{0, \ldots, r-2\}$, label the elements of each cocircuit as $C^*_i = \{x_{i,1}, \ldots, x_{i,k}\}$. For each $j$ in $\{1, \ldots, k\}$, define $B_j = \{x_{0,j}, x_{1,j}, \ldots, x_{r-2,j}, e\}$. To see that $B_j$ is independent, recall that $C^*_i$ is a cocircuit of $M|F_i$. It follows that $r_{M|F_i}(F_{i+1}) = r_{M|F_i}(F_i) - 1$, and no element of $C^*_i$ belongs to the closure of $F_{i+1}$ in $M$. Consequently, iteratively selecting exactly one element from each $C^*_i$ sequentially from $i = r-1$ down to $0$ builds an independent set. Thus, $B_j$ is independent, and as $|B_j| = r$, it is a basis. Furthermore, because the sets $C^*_i$ are disjoint, the bases $B_1, \ldots, B_k$ share no elements other than $e$.

If $M$ is a general maximal $k$-degenerate matroid, then $|C^*_{r-1}| = k$. We label its elements $\{x_{r-1, 1}, \ldots, x_{r-1, k}\}$ and define $B_j = \{x_{0,j}, \ldots, x_{r-1, j}\}$. By the same logic, these $B_j$ form $k$ pairwise disjoint bases. 
\end{proof}

Next we note some consequences of the size bound established in Proposition~\ref{prop: size_bound}. 

A matroid $M$ is \textit{minimally connected} if $M$ is connected but $M \ba e$ is disconnected for each element $e$ of $M$. Similarly, a matroid $M$ is \textit{minimally $3$-connected} if $M$ is $3$-connected and $M \ba e$ is not $3$-connected for every element $e$ of $M$. The following result of Seymour \cite{seymour} establishes that a minimally connected matroid has an abundance of cocircuits of size two. 

\begin{proposition}
\label{seymour_2-cocircuits}
Let $M$ be a minimally connected matroid. Then $M$ has a cobasis each element of which is in a $2$-element cocircuit. 
\end{proposition}

Oxley \cite{Oxl81} proved a similar result for minimally $3$-connected matroids. 

\begin{proposition}
\label{triads}
Let $M$ be a minimally $3$-connected matroid. Then $M$ has a cobasis each element of which is in a $3$-element cocircuit. 
\end{proposition}

Using the slicing decomposition, we note the following as a consequence of Propositions \ref{seymour_2-cocircuits} and \ref{triads}.  

\begin{proposition}
\label{minimal_connected_degenerate}
Let $M$ be a minimally connected matroid. Then $M$ is $2$-degenerate. Furthermore, if $M$ is minimally $3$-connected, then $M$ is $3$-degenerate. 
\end{proposition}

A straightforward consequence of Proposition~\ref{prop: size_bound} is that minimally connected and minimally $3$-connected matroids have at most $2r-1$ and $3r-2$ elements, respectively. Note that, for sufficiently large $r$, these are weaker than the tight bounds of $2r-2$ and $3r-6$ proved by Murty \cite{murty_74} and Oxley \cite{Oxl81}, respectively. 

We turn now to the class of matroids $M$ such that both $M$ and $M^*$ are $k$-degenerate. We call such matroids \textit{bi-$k$-degenerate}. We omit the elementary proof of the following.

\begin{proposition}
\label{prop: bi-k-degen-char}
The following are equivalent for a matroid $M$:
\begin{enumerate}[label=(\roman*)]
    \item $M$ is bi-$k$-degenerate.
    \item For every proper subset $T$ of $E(M)$, the deletion $M \ba T$ has a cocircuit of size at most $k$, and the contraction $M / T$ has a circuit of size at most $k$. 
    \item The ground set $E(M)$ can be reduced to the empty set by successively deleting a cocircuit of size at most $k$, and also by successively contracting a circuit of size at most $k$. 
\end{enumerate}
\end{proposition}

We use the following lemma to show that maximal $k$-degenerate matroids are bi-$k$-degenerate.

\begin{lemma}
\label{adjoin_cocircuits}
For $t \geq 2$, let $M$ be a matroid and let $C^*$ be a cocircuit of $M$ with $|C^*| \ge 2$. If the dual of $M \ba C^*$ is $t$-degenerate, then the dual of $M$ is $t$-degenerate. 
\end{lemma}

\begin{proof}
Let $N = M \ba C^*$. Since $N^*$ is $t$-degenerate, the ground set of $N$ can be reduced to the empty set via a sequence of contractions of circuits of size at most $t$.

Performing this exact sequence of contractions on $M$ eliminates all elements of $E(N)$, leaving the minor $M / E(N)$ whose ground set is $C^*$. Note that a set is a cocircuit of a contraction minor if and only if it is a cocircuit of the original matroid disjoint from the contracted set. Because $C^*$ is a cocircuit of $M$ disjoint from $E(N)$, it follows that $C^*$ is a cocircuit of $M/E(N)$. Therefore, $M / E(N) \cong U_{1, |C^*|}$. Since $|C^*| \ge 2$, the ground set of $U_{1, |C^*|}$ can be reduced to the empty set by successively contracting circuits of size at most two. The result now follows.
\end{proof}

We say $M$ is obtained from $N$ by \textit{adjoining a cocircuit} $C^*$ if $N = M \ba C^*$.

\begin{corollary}
\label{size_bound_bi_k_degenerate}
For $k \geq 2$, every general maximal $k$-degenerate matroid is bi-$k$-degenerate. For $k \geq 3$, every simple maximal $k$-degenerate matroid of rank at least two is bi-$k$-degenerate. 
\end{corollary}

\begin{proof}
For $k \geq 2$, the unique general maximal $k$-degenerate matroid of rank one, $U_{1,k}$, is bi-$k$-degenerate. Since every general maximal $k$-degenerate matroid is obtained from $U_{1,k}$ by successively adjoining cocircuits of size $k$, the result for general maximal matroids now follows by repeatedly applying Lemma~\ref{adjoin_cocircuits} with $t=k$. 

For $k \geq 3$, the unique simple maximal $k$-degenerate matroid of rank two is $U_{2,k+1}$, which is bi-$k$-degenerate. Therefore the result similarly follows by Lemma~\ref{adjoin_cocircuits} with $t=k$.
\end{proof}

It follows from Corollary~\ref{size_bound_bi_k_degenerate} that the dual of a general maximal $2$-degenerate matroid is $2$-degenerate, and therefore $3$-degenerate. We prove the following strengthening of Corollary~\ref{size_bound_bi_k_degenerate}.

\begin{proposition}
\label{dual_3_degenerate}
For $k \geq 3$, let $M$ be a general maximal $k$-degenerate matroid of rank at least one, or a simple maximal $k$-degenerate matroid of rank at least two. Then the dual of $M$ is $3$-degenerate. 
\end{proposition}

\begin{proof}
We proceed by induction on the rank $r$ of $M$. If $M$ is general maximal, the base case is rank $r=1$, where $M \cong U_{1,k}$. Its dual is $U_{k-1,k}$, which is $2$-degenerate (and therefore $3$-degenerate). If $M$ is simple maximal, the base case is $r=2$, where $M \cong U_{2,k+1}$. Its dual is $U_{k-1,k+1}$, which is $3$-degenerate.

Assume the result holds for rank $r-1$, and let $M$ be a maximal $k$-degenerate matroid (general or simple) of rank $r$. 
Let $C^*$ be the first set in a slicing of $M$. Then $|C^*|=k$ and $N = M \ba C^*$ is a maximal $k$-degenerate matroid of rank $r-1$. By the inductive hypothesis, $N^*$ is $3$-degenerate. Since $|C^*| = k \ge 2$, by Lemma~\ref{adjoin_cocircuits} with $t=3$, it immediately follows that $M^*$ is $3$-degenerate.
\end{proof}

\section{\texorpdfstring{Minimally $k$-degenerate of extremal size}{Minimally k-degenerate of extremal size}}\label{sec: simple}

In this section, we first prove Theorems~\ref{thm: simple min 3-degen} and \ref{thm: simple min k-degen}. Recall that for $k \geq 2$, a matroid $M$ is minimally $k$-degenerate if it has cogirth $k$ and every proper restriction of $M$ has cogirth less than $k$. Because the deletion of a single element can reduce the cogirth by at most one, this is equivalent to simply requiring that $M$ has cogirth at least $k$, while every proper restriction of $M$ has cogirth strictly less than $k$.

We begin with the following observation.

\begin{lemma}\label{lem: contraction is mkd}
    Suppose $M=M_1\oplus_2 M_2$ at basepoint $p$ is a minimally $k$-degenerate matroid for some $k\geq 2$ and $E(M_1)\geq 3$. If $g^*(M_2) \geq k$, then $M_1/p$ is minimally $k$-degenerate.
\end{lemma}

\begin{proof}
    Recall that (see, for example,~\cite[Proposition~7.1.20]{Oxl11})
    \begin{align*}
        \cC^*(M)
        ={}&\cC^*(M_1/p)\cup\cC^*(M_2/p)\\
        &\cup\bigl\{(C^*\cup D^*)\backslash\{p\}:
        p\in C^*\in\cC^*(M_1)\ \text{and}\
        p\in D^*\in\cC^*(M_2)\bigr\}.
    \end{align*}

Therefore, $g^*(M_1/p) \geq k$. Suppose that $N_1=(M_1/p)\ba X$ is a nonempty proper restriction of $M_1/p$ such that $g^*(N_1)\geq k$. If $p$ is a coloop of $M_1\ba X$, then $N_1=(M_1\ba X)\ba p$, which is a nonempty proper restriction of $M$, a contradiction. Therefore, $(M_1\ba X)\oplus_2 M_2$ is well-defined and, as a nonempty proper restriction of $M$, has a cocircuit of size less than $k$.

    The cocircuits of $(M_1\ba X)\oplus_2 M_2$ are the members of
    \begin{align*}
        &\cC^*(N_1)\cup\cC^*(M_2/p)\\
        &\quad\cup\bigl\{(C^*\cup D^*)\backslash\{p\}:
        p\in C^*\in\cC^*(M_1\ba X)\ \text{and}\
        p\in D^*\in\cC^*(M_2)\bigr\}.
    \end{align*}
    Every member of $\cC^*(N_1)$ has size at least $k$. Moreover, since every cocircuit of $M_2/p$ is a cocircuit of $M_2$ that avoids $p$, every member of $\cC^*(M_2/p)$ has size at least $k$.

    It remains to consider the mixed cocircuits. Choose $C^*\in \cC^*(M_1\ba X)$ and $D^*\in \cC^*(M_2)$ such that $p \in C^*$ and $p \in D^*$. Since $p$ is not a coloop of $M_1\ba X$, we have $|C^*\backslash\{p\}|\geq 1$. Since $g^*(M_2)\geq k$, we also have $|D^*\backslash\{p\}|\geq k-1$. Furthermore, because $M_1$ and $M_2$ intersect only at the basepoint $p$, the sets $C^*\backslash\{p\}$ and $D^*\backslash\{p\}$ are disjoint. Therefore,
    \[
        \bigl|(C^*\cup D^*)\backslash\{p\}\bigr| = |C^*\backslash\{p\}| + |D^*\backslash\{p\}| \geq 1+(k-1)=k.
    \]
    
    Thus, $(M_1\ba X)\oplus_2 M_2$ has no cocircuit of size less than $k$, a contradiction. Therefore, $M_1/p$ is minimally $k$-degenerate.
\end{proof}

Note that by Proposition~\ref{prop: size_bound}, it follows that a simple minimally $k$-degenerate matroid $M$ has size at most $(k-1)(r(M)-1)+2$.
We further note the following.

\begin{lemma}
\label{lem: 2-sum_components_extremal}
For $k \ge 3$, suppose $M = M_1 \oplus_2 M_2$ at basepoint $p$, where $M$ is a simple minimally $k$-degenerate matroid of extremal size $|E(M)| = (k-1)(r(M)-1)+2$. Then both $M_1$ and $M_2$ are simple minimally $k$-degenerate matroids of extremal size.
\end{lemma}

\begin{proof}
Let
\[
    f(s)=(k-1)(s-1)+2.
\]
For each $i\in\{1,2\}$, the matroid $M_i\ba p$ is a restriction of $M$, and hence is simple and $(k-1)$-degenerate. Moreover, at least one of $M_1$ and $M_2$ is simple. Indeed, since $M_i\ba p$ is simple, every parallel pair of $M_i$ must contain $p$. If both $M_1$ and $M_2$ had an element parallel to $p$, then those two elements would be parallel in $M$, contradicting the simplicity of $M$. Without loss of generality, we may assume that $M_1$ is simple.

Suppose that $M_1$ is $(k-1)$-degenerate. By Proposition~\ref{prop: size_bound},
\[
    |E(M_1)|\leq (k-1)(r(M_1)-1)+1=f(r(M_1))-1.
\]
Since $p$ is not a coloop of $M_2$, we have $r(M_2\ba p)=r(M_2)$. Applying Proposition~\ref{prop: size_bound} to the simple $(k-1)$-degenerate matroid $M_2\ba p$ gives
\[
    |E(M_2)|-1\leq (k-1)(r(M_2)-1)+1,
\]
and hence $|E(M_2)|\leq f(r(M_2))$. Therefore,
\begin{align*}
    |E(M)|
    &=|E(M_1)|+|E(M_2)|-2\\
    &\leq f(r(M_1))+f(r(M_2))-3\\
    &=f(r(M))-1,
\end{align*}
contradicting $|E(M)|=f(r(M))$. Thus, $M_1$ is not $(k-1)$-degenerate. Consequently, $M_1$ has a restriction $N_1$ such that $g^*(N_1)\geq k$. Since $M_1\ba p$ is $(k-1)$-degenerate, we have $p\in E(N_1)$.

We next show that $M_2$ is simple. Suppose otherwise. Then there is an element $e\in E(M_2)-\{p\}$ that is parallel to $p$ in $M_2$. The matroid $M$ has a proper restriction isomorphic to the matroid obtained from $N_1$ by relabeling $p$ as $e$. This restriction has cogirth at least $k$, contradicting the minimal $k$-degeneracy of $M$. Therefore, $M_2$ is simple.

By symmetry, the preceding counting argument shows that $M_2$ is not $(k-1)$-degenerate. Hence $M_2$ has a restriction $N_2$ such that $g^*(N_2)\geq k$, and necessarily $p\in E(N_2)$. Since $g^*(N_i) \geq k$ for each $i \in \{1,2\}$, the basepoint $p$ is not a coloop of $N_i$. Therefore, $N_1\oplus_2 N_2$ is well-defined.

Recall that the cocircuits of $N_1\oplus_2 N_2$ are the members of
\begin{align*}
    &\cC^*(N_1/p)\cup\cC^*(N_2/p)\\
    &\quad\cup\bigl\{(C_1^*\cup C_2^*)\backslash\{p\}:
    p\in C_1^*\in\cC^*(N_1)\ \text{and}\
    p\in C_2^*\in\cC^*(N_2)\bigr\}.
\end{align*}
Every member of $\cC^*(N_1/p)\cup\cC^*(N_2/p)$ has size at least $k$. Moreover, if $p\in C_i^*\in\cC^*(N_i)$ for each $i\in\{1,2\}$, then
\[
    \bigl|(C_1^*\cup C_2^*)\backslash\{p\}\bigr|
    =|C_1^*\backslash\{p\}|+|C_2^*\backslash\{p\}|
    \geq 2k-2\geq k.
\]
Therefore, $g^*(N_1\oplus_2 N_2)\geq k$. Since $N_1\oplus_2 N_2$ is a restriction of $M$, the minimality of $M$ implies that $N_1=M_1$ and $N_2=M_2$. In particular, $M_1$ and $M_2$ are simple matroids with cogirth at least $k$.

We now show that $M_1$ is minimally $k$-degenerate. Suppose that $R_1$ is a nonempty proper restriction of $M_1$ such that $g^*(R_1)\geq k$. If $p\notin E(R_1)$, then $R_1$ is a proper restriction of $M$, a contradiction. Thus, $p\in E(R_1)$. Since $g^*(R_1)\geq k$, the element $p$ is not a coloop of $R_1$, so $R_1\oplus_2 M_2$ is well-defined. The cocircuits of $R_1\oplus_2 M_2$ are the members of
\begin{align*}
    &\cC^*(R_1/p)\cup\cC^*(M_2/p)\\
    &\quad\cup\bigl\{(C_1^*\cup C_2^*)\backslash\{p\}:
    p\in C_1^*\in\cC^*(R_1)\ \text{and}\
    p\in C_2^*\in\cC^*(M_2)\bigr\}.
\end{align*}
Every member of the first two families has size at least $k$, and every member of the third family has size at least $2k-2$. Therefore, $g^*(R_1\oplus_2 M_2)\geq k$. However, $R_1\oplus_2 M_2$ is a proper restriction of $M$, a contradiction. Therefore, $M_1$ is minimally $k$-degenerate. By symmetry, $M_2$ is also minimally $k$-degenerate.

Applying the size bound to $M_1$ and $M_2$ gives $|E(M_i)|\leq f(r(M_i))$ for each $i\in\{1,2\}$. Since
\[
    r(M)=r(M_1)+r(M_2)-1
\]
and
\[
    f(r(M))=f(r(M_1))+f(r(M_2))-2,
\]
we have
\begin{align*}
    f(r(M_1))+f(r(M_2))-2
    &=|E(M)|\\
    &=|E(M_1)|+|E(M_2)|-2.
\end{align*}
It follows that
\[
    |E(M_i)|=f(r(M_i))=(k-1)(r(M_i)-1)+2
\]
for each $i\in\{1,2\}$. Thus, both $M_1$ and $M_2$ are simple minimally $k$-degenerate matroids of extremal size.
\end{proof}

For a minimally $3$-degenerate matroid $M$, an element $e$ of $M$ is \textit{contractible} if $M/e$ is still minimally $3$-degenerate. We are now ready to prove Theorem~\ref{thm: simple min 3-degen}. 

\begin{proof}[Proof of Theorem~\ref{thm: simple min 3-degen}]
    The size bound is an immediate consequence of Proposition~\ref{prop: size_bound}. It remains to characterize the matroids for which equality holds. We proceed by induction on $|E(M)|$. The statement is readily verified when $|E(M)|\leq 4$. Now let $M$ be a simple, minimally $3$-degenerate matroid satisfying $|E(M)|=2r(M)$, and assume that the characterization holds for every simple, minimally $3$-degenerate matroid $N$ such that $|E(N)|<|E(M)|$ and $|E(N)|=2r(N)$.
    
    Because $M$ is minimally $3$-degenerate, it is connected. Otherwise a component of $M$ would be a proper nonempty restriction of $M$ with cogirth at least three. If $M$ is $3$-connected, then every restriction of $M$ with at least four elements is not $3$-connected. It follows from~\cite[Theorem 1.2]{sm3c_matroid} that $M$ is either a wheel or a whirl. So we may now assume that $M=M_1\oplus_2 M_2$, for some $M_1$ and $M_2$ such that $|E(M_1)|\geq 3$, $|E(M_2)|\geq 3$, and $E(M_1)\cap E(M_2)=\{p\}$. By Lemma~\ref{lem: 2-sum_components_extremal}, both $M_1$ and $M_2$ are simple minimally $3$-degenerate matroids of extremal size. Therefore, by induction, both $M_1$ and $M_2$ are matroids characterized in (\romannum{1})-(\romannum{3}).

    Further, by Lemma~\ref{lem: contraction is mkd}, $M_i/p$ is minimally $3$-degenerate for $i\in\{1,2\}$, that is, $p$ is contractible in both $M_1$ and $M_2$. A direct check shows that an extremal matroid has a contractible element only when it is a wheel or a whirl, and that the contractible elements of a wheel or whirl are precisely its spokes. It follows that each $M_i$ is a wheel or a whirl and that $p$ is a spoke of both. Thus $M$ is characterized as~(\romannum{3}).
\end{proof}

Next we prove Theorem~\ref{thm: simple min k-degen}. We use the following result.

\begin{lemma}
\label{lem: cocircuit_intersection}
For $k \geq 3$, let $M$ be a simple minimally $k$-degenerate matroid such that $r(M) \geq 3$ and $|E(M)| = (k-1)(r(M)-1)+2$. Suppose that $C^*$ and $D^*$ are distinct cocircuits of $M$ of size $k$. Then $|C^* \cap D^*| \leq 1$. 
\end{lemma}

\begin{proof}
Assume to the contrary that $\{e,f\} \subseteq C^* \cap D^*$. Observe that $M\ba e$ is a simple maximal $(k-1)$-degenerate matroid so each set in a slicing of $M\ba e$ except the last set has size exactly equal to $k-1$. We choose a slicing of $M\ba e$ such that $C^*-e$ is the first set. Since $D^*-C^*$ contains a cocircuit of $M \ba e \ba (C^*-e)= M\ba C^*$, the second set in the slicing can be chosen to be contained in $D^*-C^*$, which has size at most $k-2$. This is a contradiction because $r(M) \geq 3$ so the second set in the slicing is not the last set.
\end{proof}

\begin{proof}[Proof of Theorem~\ref{thm: simple min k-degen}]
The inequality follows from Proposition~\ref{prop: size_bound}. We characterize the matroids that attain equality by induction on $|E(M)|$. Suppose that
\[
    |E(M)|=(k-1)(r(M)-1)+2.
\]
If $r(M)=2$, then the simplicity of $M$ gives $M\cong U_{2,k+1}$. We may therefore assume that $r(M)\geq 3$.

We first consider the case in which $M$ is $3$-connected. We claim that every element of $M$ is contained in a cocircuit of size $k$. Let $e\in E(M)$. Since every nonempty restriction of $M\ba e$ is a proper restriction of $M$, the matroid $M\ba e$ is $(k-1)$-degenerate. Moreover, deletion reduces cogirth by at most one, so $g^*(M\ba e)=k-1$. Let $C_e^*$ be a cocircuit of $M\ba e$ of size $k-1$. There is a cocircuit $C^*$ of $M$ such that $C^*\backslash\{e\}=C_e^*$. Since $g^*(M)=k$, we must have $e\in C^*$, and hence $|C^*|=k$.

Fix an element $e\in E(M)$. Since $e$ is not a coloop, $r(M\ba e)=r(M)$. Furthermore,
\[
    |E(M\ba e)|=(k-1)(r(M)-1)+1.
\]
Thus, $M\ba e$ attains equality in the size bound for simple $(k-1)$-degenerate matroids. Applying the equality argument in Proposition~\ref{prop: size_bound}, we may successively delete $r(M)-2$ cocircuits, each of size $k-1$, leaving a simple rank-$2$ restriction with exactly $k$ elements. Therefore, $M$ has a restriction isomorphic to $U_{2,k}$. Let its ground set $\{e_1,\ldots,e_k\}$ be $L$. Suppose that a cocircuit $C^*$ of $M$ intersects $L$. We claim that
\[
    |C^*\cap L|\geq k-1.
\]
Suppose otherwise. Then there are distinct elements $x,y\in L-C^*$. Choose $z\in C^*\cap L$. Since $M|L\cong U_{2,k}$, the set $\{x,y,z\}$ is a circuit of $M$. However, $\{x,y,z\}\cap C^*=\{z\}$, contradicting circuit-cocircuit orthogonality.

Choose a cocircuit $C^*$ of size $k$ containing $e_1$. It follows that $C^*$ contains either all of $L$ or exactly $k-1$ elements of $L$. Suppose first that $L\subseteq C^*$. Since $|C^*|=|L|=k$, we have $C^*=L$. As $C^*$ is a cocircuit, $r(E(M)-C^*)=r(M)-1$. Since $r(C^*)=2$, we have
\[
    \lambda_M(C^*)=r(C^*)+r(E(M)-C^*)-r(M)=1.
\]
Because $M$ is $3$-connected and $|C^*|\geq 2$, it follows that $|E(M)-C^*|\leq 1$. But $E(M)-C^*$ is a hyperplane, so this implies $r(M)\leq 2$, contradicting our assumption that $r(M)\geq 3$.

Therefore, $C^*$ contains exactly $k-1$ elements of $L$. We may assume that
\[
    C^*\cap L=L-\{e_k\}.
\]
Let $D^*$ be a cocircuit of size $k$ containing $e_k$. By the same orthogonality argument, $|D^*\cap L|\geq k-1$. Since $e_k\notin C^*$ and $e_k\in D^*$, the cocircuits $C^*$ and $D^*$ are distinct. Moreover,
\[
    |C^*\cap D^*|\geq k-2\geq 2,
\]
contradicting Lemma~\ref{lem: cocircuit_intersection}. Thus, there is no $3$-connected extremal example of rank at least three. Consequently, the only $3$-connected equality case is $U_{2,k+1}$.

We may now assume that $M$ is not $3$-connected. The matroid $M$ is connected, since otherwise one of its components would be a proper restriction with cogirth at least $k$. Therefore, $M$ has a nontrivial $2$-sum decomposition $M=M_1\oplus_2 M_2$ at a basepoint $p$. By Lemma~\ref{lem: 2-sum_components_extremal}, both $M_1$ and $M_2$ are simple minimally $k$-degenerate matroids satisfying
\[
    |E(M_i)|=(k-1)(r(M_i)-1)+2.
\]
Since the $2$-sum is nontrivial, each $M_i$ has fewer elements than $M$. The induction hypothesis therefore implies that each $M_i$ is isomorphic to either $U_{2,k+1}$ or $U_{2,k+1}\oplus_2 U_{2,k+1}$.

By Lemma~\ref{lem: contraction is mkd}, both $M_1/p$ and $M_2/p$ are minimally $k$-degenerate. Let $Q$ be a matroid isomorphic to $U_{2,k+1}\oplus_2 U_{2,k+1}$. We show that $Q$ has no element whose contraction is minimally $k$-degenerate.

Let $A$ and $B$ be the two $k$-element sides of $Q$, and let $e\in A$. In $Q/e$, the set $A-\{e\}$ is a parallel class of size $k-1$, while every element of $B$ is in a singleton parallel class. Choose $a\in A-\{e\}$. In $(Q/e)\ba a$, the set $A-\{e,a\}$ is a parallel class of size $k-2$, while the elements of $B$ remain in singleton parallel classes. Hence the cocircuits of $(Q/e)\ba a$ obtained as complements of these parallel classes have sizes $k$ and $2k-3$. Since $2k-3\geq k$, we have $g^*((Q/e)\ba a)=k$. Thus, $Q/e$ has a proper restriction of cogirth $k$ and is not minimally $k$-degenerate. The case $e\in B$ is symmetric. Therefore, no element of $Q$ has the required contraction property.

Since $M_i/p$ is minimally $k$-degenerate for each $i\in\{1,2\}$, neither $M_1$ nor $M_2$ can be isomorphic to $Q$. It follows that $M_1\cong M_2\cong U_{2,k+1}$, and hence
\[
    M\cong U_{2,k+1}\oplus_2 U_{2,k+1}.
\]

We omit the routine proof that both $U_{2,k+1}$ and $U_{2,k+1} \oplus_2 U_{2,k+1}$ are minimally $k$-degenerate.
\end{proof}

We now prove Theorem~\ref{thm: non-simple min k-degen}. We note the following lemmas.

\begin{lemma}
\label{lem: parallel_class}
For $k \geq 3$, let $M$ be a minimally $k$-degenerate matroid that is not isomorphic to $U_{1,k}$, and let $\{e_1, \ldots, e_{k-1}\}$ be a maximal set of parallel elements of $M$. Then $M/e_1\ba \{e_2, \ldots, e_{k-1}\}$ is minimally $k$-degenerate.
\end{lemma}

\begin{proof}
Since $\cl_M(e_1)=\{e_1, \ldots, e_{k-1}\}$, it is clear that $M/e_1\ba \{e_2, \ldots, e_{k-1}\}$ is loopless. Because $g^*(M/e_1)\geq g^*(M)=k$ and $\{e_2, \ldots, e_{k-1}\}$ are loops in $M/e_1$, it follows that $g^*(M/e_1\ba \{e_2, \ldots, e_{k-1}\})\geq k$. 

Suppose that $M/e_1\ba \{e_2, \ldots, e_{k-1}\}$ has a proper nonempty restriction $N$ with $g^*(N)\ge k$. Then $\{e_1,\ldots, e_{k-1}\}\subseteq \cl_M(E(N))$; otherwise, $e_1 \notin \cl_M(E(N))$ and so $N$ is isomorphic to the restriction $M|E(N)$, which is a proper restriction of $M$ with cogirth at least $k$, a contradiction. 

Let $N'=M|(E(N)\cup\{e_1, \ldots, e_{k-1}\})$. Since $N'$ is a proper nonempty restriction of $M$, the matroid $N'$ has a cocircuit $C^*$ of size at most $k-1$. Note that $N = N'/e_1\ba \{e_2, \ldots, e_{k-1}\}$, so if $C^*\cap \{e_1, \ldots, e_{k-1}\}=\emptyset$, then $C^*$ is also a cocircuit of $N$, a contradiction. Moreover, since $\{e_1, \ldots, e_{k-1}\}$ is a parallel class of $N'$, every cocircuit of $N'$ that intersects this class must contain all of it. Thus, $\{e_1, \ldots, e_{k-1}\} \subseteq C^*$. Because $|C^*| \leq k-1$, we must have exactly $C^*=\{e_1, \ldots, e_{k-1}\}$. 

Because $\{e_1,\ldots, e_{k-1}\}\subseteq \cl_M(E(N))$, there is a circuit $C$ of $N'$ such that $e_1 \in C\subseteq E(N)\cup\{e_1\}$. By orthogonality, the intersection of the circuit $C$ and the cocircuit $C^*$ cannot be exactly one element. However, $C \cap C^* = \{e_1\}$, which is a contradiction. 
\end{proof}

\begin{lemma}
\label{lem: series connection}
    Suppose $M_1$ and $M_2$ are minimally $k$-degenerate matroids where $k\geq 2$ and $E(M_1)\cap E(M_2)=\{p\}$. Then their series connection $S(M_1,M_2)$ is also minimally $k$-degenerate.
\end{lemma}

\begin{proof}
    Let $S=S(M_1,M_2)$. Since $g^*(M_i)=k\geq 2$, neither $M_i$ has a coloop. Thus the series connection at $p$ is well-defined.

    The cocircuits of $S$ are precisely the cocircuits of $M_1$ and $M_2$ together with the sets $C_1^*\cup C_2^*-\{p\}$ where $C_i^*$ is a cocircuit of $M_i$ containing $p$. As $|C^*|\geq k\geq 2$ for every cocircuit $C^*$ in either $M_1$ or $M_2$, it follows that $g^*(S)\geq k$. Moreover, each $M_i$ has a cocircuit with exactly $k$ elements, and it is also a cocircuit of $S$. Consequently, $g^*(S)=k$.

    To see that $S$ is minimally $k$-degenerate, suppose that $A$ is a non-empty proper subset of $E(S)$ such that $g^*(S|A)\geq k$. Let $A_i$ be the intersection of $A$ and $E(M_i)$ for $i\in\{1,2\}$. We may assume that $A_1$ is a proper non-empty subset of $E(M_1)$. By assumption, $A_1$ contains a cocircuit $D^*$ of $M_1|A_1$ with $|D^*|<k$, and $D^*$ is also a cocircuit of $S$. Therefore, the set $D^*$ is a union of cocircuits of $S|A$, and hence $S|A$ has a cocircuit of size less than $k$, a contradiction.
\end{proof}

We are now ready to prove Theorem~\ref{thm: non-simple min k-degen}. 

\begin{proof}[Proof of Theorem~\ref{thm: non-simple min k-degen}]
It is elementary to see that minimally $2$-degenerate matroids are precisely $U_{r,r+1}$ for some $r\geq 1$, and each $U_{r,r+1}$ can be obtained via iterated series connections of $U_{1,2}$. Therefore, we assume that $k\geq 3$.

Since $M \ba e$ is $(k-1)$-degenerate for every element $e \in E(M)$, it follows by Proposition~\ref{prop: size_bound} that $|E(M \ba e)| \leq (k-1)r(M \ba e)$. Because $M$ has cogirth $k \geq 3$, it has no coloops, so $r(M \ba e) = r(M)$. Therefore, $|E(M)| - 1 \leq (k-1)r(M)$ and so $|E(M)| \leq (k-1)r(M) + 1$. 

Suppose that $M$ has size equal to $(k-1)r(M)+1$. We proceed via induction on the rank. If $r(M) =1$, then $M \cong U_{1,k}$ and the result holds. Assume $r(M) > 1$. Observe that for every element $e$ in $E(M)$, the deletion $M \ba e$ attains the maximum size for a $(k-1)$-degenerate matroid. Since the final set in a slicing of $M \ba e$ is a rank-$1$ flat of size $k-1$, it follows that $M \ba e$ and, therefore $M$, contains a parallel class of size $k-1$. Furthermore, this parallel class cannot have size $k$ in $M$; otherwise, $M$ would contain a proper restriction isomorphic to $U_{1,k}$, a contradiction.

It follows by Lemma~\ref{lem: parallel_class} that the matroid $M/e_1 \ba \{e_2, \ldots, e_{k-1}\}$ is minimally $k$-degenerate. Since $M/e_1 \ba \{e_2, \ldots, e_{k-1}\}$ has rank $r(M)-1$ and size $(k-1)(r(M)-1) + 1$, by induction, it can be built from $U_{1,k}$ by repeatedly applying series connection. 

Because $M$ is minimally $k$-degenerate, $M \ba e_1$ contains a cocircuit of size at most $k-1$. This cocircuit must be of the form $C^* - e_1$, where $C^*$ is a cocircuit of $M$ containing $e_1$. Since $M$ has cogirth $k$, it follows that $|C^*| = k$. Because $\{e_1, \ldots, e_{k-1}\}$ is a set of parallel elements in $M$, any cocircuit intersecting this set must contain all of it; thus, $\{e_1, \ldots, e_{k-1}\} \subseteq C^*$. Let $C^* = \{e_1, \ldots, e_{k-1}, f\}$. 

Note that this cocircuit $C^*$ is unique. Suppose that there is another cocircuit $D^* = \{e_1, \ldots, e_{k-1}, g\}$ of size $k$. By cocircuit elimination on $e_1$, there is a cocircuit $E^* \subseteq (C^* \cup D^*) - \{e_1\} = \{e_2, \ldots, e_{k-1}, f, g\}$. Since $E^*$ does not contain $e_1$, it is disjoint from the parallel class $\{e_1, \ldots, e_{k-1}\}$ so $E^* \subseteq \{f, g\}$. Since $M$ has cogirth $k \geq 3$, this is a contradiction.

Because $C^* = \{e_1, \ldots, e_{k-1}, f\}$ is the unique cocircuit containing the parallel class $\{e_1, \ldots, e_{k-1}\}$ of $M$, it now follows that $M$ is obtained from $M/e_1 \ba \{e_2, \ldots, e_{k-1}\}$ by a series connection with a $U_{1,k}$ having ground set $\{e_1, \ldots, e_{k-1}, f\}$ and basepoint $f$.

By Lemma~\ref{lem: series connection}, matroids obtained via iterated series connections of $U_{1,k}$ are minimally $k$-degenerate. One may easily check that they attain the maximum size bound.
\end{proof}

\section{Arboricity and Cocircumference}\label{arboricity}

Recall that the arboricity, $a(M)$, of a matroid $M$ is the minimum number of bases required to cover $E(M)$, or equivalently, the minimum number of independent sets into which $E(M)$ can be partitioned. The sizes of the largest circuit and the largest cocircuit of $M$ are denoted by $c(M)$ and $c^*(M)$, respectively. 

Minh and Trung~\cite[Theorem~3.3]{nguyen} showed the following.

\begin{proposition}
\label{symbolic_result}
Let $M$ be a matroid. Then $a(M) \leq c^*(M)$.
\end{proposition}

We strengthen this bound using degeneracy $d(M)$ as follows. 

\begin{proposition}
\label{prop: arboricity_bound}
Let $M$ be a matroid. Then $a(M) \le d(M) \le c^*(M)$.
\end{proposition}

\begin{proof}
We proceed by induction on rank. A $k$-degenerate rank-$1$ matroid is $U_{1,s}$ for some $s\leq k$, and hence can be covered by $k$ bases. We therefore assume that $r(M)\geq 2$ and the statement holds for every $k$-degenerate matroid $N$ with $r(N)<r(M)$. Let $C^* = \{e_1, \ldots, e_t\}$ be a cocircuit of $M$ of minimum size. Note that $t \leq k$ and $M' = M \ba C^*$ is $k$-degenerate. Since $M'$ has rank $r(M)-1$, by induction, $E(M')$ can be covered by $k$ bases $B_1, \dots, B_k$. For each $i$ in $\{1, \ldots, t\}$, we let $B_i' = B_i \cup \{e_i\}$. For the remaining indices $i$ from $t+1$ to $k$, if any, let $B_i' = B_i \cup \{e_1\}$. Note that $B_1', \ldots, B_k'$ are bases of $M$ that cover $E(M)$. Therefore $a(M) \leq k$. Since the degeneracy of a matroid does not exceed the size of the largest cocircuit, the result follows.
\end{proof}

The next result determines precisely when degeneracy equals the maximum cocircuit size. It is well known that for connected graphs, degeneracy equals the maximum degree if and only if the graph is regular. The following is its matroid analogue.

\begin{proposition}
\label{degeneracy_equals_cogirth}
Let $M$ be a connected matroid. Then the degeneracy of $M$ equals the size of a largest cocircuit of $M$ if and only if all cocircuits of $M$ have the same size.
\end{proposition}

\begin{proof}
Let $c$ be the size of a largest cocircuit of $M$. If all cocircuits have size $c$, it is clear that $d(M)=c$. Conversely, suppose that $M$ is a connected matroid of minimal rank such that $d(M)$ equals the largest cocircuit size $c$, but $M$ has a cocircuit $C^*$ of size $|C^*| < c$. Let $M' = M \ba C^*$. By the slicing characterization of degeneracy, if $d(M') \le c-1$, then $d(M) \le c-1$, a contradiction. Therefore, $d(M') = c$.

Because $d(M') = c$, there exists a connected component $N$ of $M'$ such that $d(N) = c$. Since $N$ is a restriction of $M$, its largest cocircuit size $c^*(N)$ is at most $c$. Since $c = d(N) \le c^*(N) \le c$, we must have $c^*(N) = c$. Thus, by the minimality of the rank of $M$, all cocircuits of $N$ have exactly $c$ elements. 

Let $e$ be an element of $E(N)$, and let $f$ be an element of $C^*$. Because $M$ is connected, there is a cocircuit $D^*$ of $M$ containing both $e$ and $f$. The restriction of $N$ to $D^* \cap E(N)$ is a union of cocircuits of $N$. Since $e \in D^* \cap E(N)$, this intersection is non-empty and thus contains at least one cocircuit of $N$, so $|D^* \cap E(N)| \geq c$. Since the size of $D^*$ is at most $c$, it follows that $D^* \subseteq E(N)$. However, $f \in D^* - E(N)$, a contradiction. 
\end{proof}

Note that the matroids appearing in Proposition~\ref{degeneracy_equals_cogirth} are the duals of equicardinal matroids~\cite{murty_71}.

The following is a matroid analogue of~\cite[Theorem~69]{bickle} and generalizes
Proposition~\ref{prop: arboricity_bound}. 

\begin{proposition}
\label{arboricity_general}
Let $M$ be a $k$-degenerate matroid of rank $r$. For positive integers $t_1, \ldots, t_m$ such that $t_1 + \cdots + t_m = k$, the ground set $E(M)$ can be partitioned into sets $T_1, \ldots, T_m$ such that $M|T_i$ is $t_i$-degenerate for every non-empty $T_i$. 
\end{proposition}

\begin{proof}
Since $M$ is $k$-degenerate, there is a slicing $C_0^*, \ldots, C_{r-1}^*$ which partitions $E(M)$ such that each $C_j^*$ is a cocircuit of $M \ba (C_0^* \cup \ldots \cup C_{j-1}^*)$ of size at most $k$.

Because $\sum_{i=1}^m t_i = k$, we can partition each cocircuit $C_j^*$ into $m$ disjoint sets $C_{j,1}^*, \ldots, C_{j,m}^*$ (some of which may be empty) such that $|C_{j,i}^*| \le t_i$ for all $1 \le i \le m$. For each $i \in \{1, \dots, m\}$, define $T_i = \bigcup_{j=0}^{r-1} C_{j,i}^*$. By construction, $\{T_1, \ldots, T_m\}$ partitions the ground set $E(M)$. 

Suppose that $T_i$ is not empty. To see that $M|T_i$ is $t_i$-degenerate, we delete the sets $C_{0,i}^*, \ldots, C_{r-1,i}^*$ in order from $M|T_i$. At step $j$, the set $C_{j,i}^*$ is the intersection of $C_j^*$ with the remaining ground set. Therefore, $C_{j,i}^*$ is a union of cocircuits in that restriction. Since $|C_{j,i}^*| \le t_i$, we can remove $C^*_{j,i}$ by successively deleting cocircuits of size at most $t_i$. Performing this for all $j$ yields a valid slicing of $M|T_i$, confirming it is $t_i$-degenerate.
\end{proof}

We devote the remainder of this section to strengthening Proposition~\ref{symbolic_result}. The \textit{co-arboricity} $a^*(M)$ of a matroid $M$ is the arboricity of its dual $M^*$. Furthermore, for any element $e$ in $E(M)$, we use $c_e(M)$ and $c^*_e(M)$ to denote the \textit{$e$-circumference} and \textit{$e$-cocircumference}, the sizes of the largest circuit and largest cocircuit containing $e$, respectively.

We will require the following covering result, which builds upon a theorem of Oxley \cite[Theorem 4.3.13]{Oxl11}. While the original theorem establishes the existence and size of the covering, the additional stipulation that it can include an arbitrarily specified cocircuit follows straightforwardly from Oxley's proof.

\begin{theorem}
\label{covering}
Let $M$ be a connected matroid with at least two elements, and let $e \in E(M)$. Then there exists a family of at most $c_e(M)-1$ cocircuits of $M$, each containing $e$, whose union is $E(M)$. Furthermore, this family can be chosen to include any arbitrary cocircuit $C^*$ containing $e$. 
\end{theorem}

Next we prove Theorem~\ref{e-circumference_intro}. 

\begin{proof}[Proof of Theorem~\ref{e-circumference_intro}]
 By Theorem~\ref{covering}, we have a collection $\{C_1^*, \ldots, C_{q}^*\}$ of cocircuits each containing $e$ such that their union is $E(M)$ and $q\leq c_e(M)-1$. Since each $C_i^*$ contains $e$, the set $C_i^* - \{e\}$ is coindependent. This implies that we have a collection of $q$ coindependent sets of $M$ whose union is $E(M) - \{e\}$. To cover the remaining element $e$, we require one additional coindependent set (such as $\{e\}$). This gives a total of $q+1\leq c_e(M)$ cobases whose union is $E(M)$. Therefore $a^*(M) \leq c_e(M)$. The result now follows by matroid duality. 
\end{proof}

An element $e$ of a matroid $M$ is \textit{free} if every circuit of $M$ that contains $e$ is spanning. We prove a strengthening of Theorem~\ref{e-circumference_intro} as follows.  

\begin{theorem}
\label{cospan}
Let $M$ be a connected matroid and $e \in E(M)$. If $e$ is not free in $M$, then $a(M) \leq c^*_e(M)-1$.
\end{theorem}

\begin{proof}
We prove the dual statement: if $e$ is not free in $M^*$, then $a^*(M) \leq c_e(M) - 1$. Since $e$ is not free in $M^*$, there is a non-spanning circuit $C^*_1$ of $M^*$ that contains $e$. Since $C_1^*$ is a cocircuit of $M$, by Theorem~\ref{covering}, there is a collection $C_1^*, \ldots, C_{q}^*$ of cocircuits of $M$ such that their union is $E(M)$, each cocircuit contains $e$, and $q \leq c_e(M)-1$. Observe that each of $C_1^* - \{e\}, \ldots, C_{q}^* - \{e\}$ is coindependent in $M$ and their union is $E(M) - \{e\}$. 

Since $C_1^* - \{e\}$ is not a basis of $M^*$, we can extend it to form one, say, $B_1^*$. Let $f \in B_1^* - C_1^*$. We may assume that $f \in C_2^*$. Then the members of the collection $B_1^*, C_2^* - \{f\}, C_3^* - \{e\}, \ldots, C_{q}^* - \{e\}$ are coindependent sets of $M$. Notice that $e \notin B_1^*$ because $C_1^*$ is a circuit in $M^*$, but $e$ is covered by $C_2^* - \{f\}$ since $e \in C_2^*$ and $e \neq f$. Meanwhile, $f$ is covered by $B_1^*$. Thus, this collection of $q$ coindependent sets covers all of $E(M)$. Since $q \leq c_e(M)-1$, it follows that $a^*(M) \leq c_e(M)-1$. The result now follows by duality. 
\end{proof}

Finally we prove Theorem~\ref{thm: equality_arboricity_intro} that characterizes the matroids that satisfy the equality in Proposition~\ref{symbolic_result}.

\begin{proof}[Proof of Theorem~\ref{thm: equality_arboricity_intro}]
Suppose $M$ is a matroid such that $a(M) = c^*(M)$. By Theorem~\ref{e-circumference_intro}, we know $a(M) \le c_e^*(M) \le c^*(M)$ for every $e \in E(M)$. Therefore, we must have $a(M) = c_e^*(M)$ for every element $e$. 

By Theorem~\ref{cospan}, it follows that every element $e$ is free in $M$. Consequently, every circuit of $M$ is spanning so $M$ is uniform. It can be easily checked that $a(M) \leq c^*(M)-1$ for connected uniform matroids unless $M \cong U_{k,k+1}$ or $U_{1,k}$.
\end{proof}

\section{Consequences for graphs}\label{sec: graphs}
All graphs considered in this section are simple. We begin by recalling that a graph $G$ is \textit{minimally $k$-degenerate} if $G$ has degeneracy equal to $k$ and every proper subgraph of $G$ has degeneracy strictly less than $k$. We define \textit{minimally $k$-edge-degenerate} graphs analogously. 

Observe that a graph $G$ is minimally $k$-edge-degenerate if and only if $G$ is $k$-edge-connected, but every proper non-empty subgraph of $G$ fails to be $k$-edge-connected. This places minimally $k$-edge-degenerate graphs naturally into the broader edge-connectivity hierarchy. For example, a $k$-edge-connected graph $G$ is \textit{minimally $k$-edge-connected} if $G \ba e$ is not $k$-edge-connected for every $e \in E(G)$ (see~\cite{maurer1978}), whereas $G$ is \textit{uniformly $k$-edge-connected} if there are exactly $k$ edge-disjoint paths between every pair of distinct vertices of $G$ (see~\cite{goring2022, kingsford2025, xu2026}).

In this section, we apply our matroid results to characterize minimally $3$-degenerate graphs of extremal size. We first establish the following.

\begin{proposition}
\label{prop: min_3_degen_graphs}
Let $G$ be a graph with $n$ vertices. If $G$ is minimally $3$-degenerate, then the number of edges satisfies $|E(G)| \le 2n - 2$. This bound holds analogously if $G$ is minimally $3$-edge-degenerate.
\end{proposition}

\begin{proof}
Since $G$ is minimally $3$-degenerate, for an edge $e$ of $G$, the graph $G \ba e$ is $2$-degenerate, so by Proposition~\ref{prop: lick_white_maximal_k_degen}, the graph $G \ba e$ has at most $2n-3$ edges. Therefore $G$ has at most $2n-2$ edges. The analogous result for minimally $3$-edge-degenerate graphs follows by Proposition~\ref{prop: mader_maximal_k_degen}.  
\end{proof}

The following notes that the class of extremal minimally $3$-degenerate graphs is contained in the class of extremal minimally $3$-edge-degenerate graphs. 

\begin{lemma}
\label{lem: min_3_degen_implies_min_3_edge_degenerate}
Let $G$ be a minimally $3$-degenerate graph with $n$ vertices and $2n-2$ edges. Then $G$ is minimally $3$-edge-degenerate.
\end{lemma}

\begin{proof}
By Proposition~\ref{prop: mader_maximal_k_degen}, it is clear that $G$ is not $2$-edge-degenerate. Since $d(G)=3$, and $ed(G) \leq d(G)$, it follows that the edge degeneracy of $G$ is three. 

Since the degeneracy of every proper subgraph of $G$ is less than three, it follows that the edge degeneracy of every proper subgraph of $G$ is less than three. Therefore $G$ is minimally $3$-edge-degenerate. 
\end{proof}

Using Theorem~\ref{thm: simple min 3-degen}, we now characterize minimally $3$-degenerate graphs of extremal size.  

\begin{theorem}
\label{thm: simple min 3-degen_graphs}
Suppose $G$ is a minimally $3$-degenerate graph with $n$ vertices. Then \[|E(G)|\leq 2n-2,\]
and the equality is attained if and only if

\begin{enumerate}
\item[(\romannum{1})] $G$ is isomorphic to a wheel graph $\mathcal{W}_k$ for some $k\geq 3$, or 

\item[(\romannum{2})] $G=G_1\oplus_2 G_2$ at a basepoint $p$ where, for each $i\in \{1,2\}$, the graph $G_i \cong \mathcal{W}_{k_i}$ for some $k_i \geq 3$, and $p$ is a spoke edge in both $G_1$ and $G_2$.
\end{enumerate}
\end{theorem}

\begin{proof}
The bound $|E(G)| \le 2n-2$ follows immediately from Proposition~\ref{prop: min_3_degen_graphs}. To characterize the graphs that attain equality, suppose $G$ has exactly $2n-2$ edges. By Lemma~\ref{lem: min_3_degen_implies_min_3_edge_degenerate}, $G$ is minimally $3$-edge-degenerate. Moreover, as $G$ is connected, its cycle matroid $M(G)$ is minimally $3$-degenerate and has rank $r(M(G)) = n-1$. Since $|E(M(G))| = 2n-2 = 2r(M(G))$, the matroid $M(G)$ attains the extremal size bound for a minimally $3$-degenerate matroid. Since whirls are non-graphic matroids, the characterization of $G$ as a wheel or the $2$-sum of wheels now follows directly from Theorem~\ref{thm: simple min 3-degen}. 

The converse is straightforward to check and hence is omitted here.
\end{proof}

\section*{Acknowledgments}

The authors thank James Oxley for suggesting the study of degeneracy in matroids.

\end{document}